\documentclass{article}
\usepackage{amsmath}
\usepackage{amsthm}
\usepackage{thmtools}
\declaretheoremstyle[headfont=\normalfont]{normalhead}
\usepackage{color}
\usepackage{graphicx}
\usepackage{morefloats}
\usepackage{ amssymb }
\usepackage{textcomp}
\usepackage{url}
\usepackage{float}
\usepackage{hyperref}
\newtheoremstyle{mydef}
{\topsep}{\topsep}%
{}{}%
{\itshape}{}
{\newline}
{%
  \rule{\textwidth}{0.0pt}\\*%
  \thmname{#1}~\thmnumber{#2}\thmnote{\-\ #3}.\\*[-1.5ex]%
  \rule{\textwidth}{0.0pt}}%
\hypersetup{
   colorlinks=true,
    linkcolor=blue,
    filecolor=magenta,      
    citecolor=blue,
    urlcolor = blue
}
\begin{document}
\theoremstyle{mydef}
\newtheorem{conjecture}{Conjecture}
\newtheorem{theorem}{Theorem}
\newtheorem{question}{Question}
\newtheorem{remark}{Remark}
\newtheorem{proposal}{Proposal}
\newtheorem{lemma}{Lemma}
\newtheorem{corollary}{Corollary}

\author{Barry Brent}

\date{26 September 2026}

\title{Dynamics of multiples of zeta by exponential functions}
\maketitle
\begin{abstract}
\hskip -.2in
We study the family
of functions
$V_z:s \mapsto  \zeta(s) \exp (zs)$
with numerical experiments. 
For selected sequences $z_n=x_nu$ on rays, our experiments suggest
that suitable fixed points of $V_{z_n}$ approach nontrivial zeros of
$\zeta$ in a nearly logarithmic pattern. We formulate this behavior as
conjectures with explicit restrictions on the ray. We prove a conditional
result identifying finite limits as zeros, and give sufficient hypotheses
for the convergence of refinement diagnostics.
Substantive AI assistance in the preparation of this draft is disclosed
in Section~\ref{sec:ai-disclosure}.
\end{abstract}

\bibliographystyle{plain}

\section{\sc introduction}
In this article, 
we study the family
of functions
$V_z:s \mapsto  \zeta(s) \exp (zs)$
with numerical experiments. 
For selected sequences $z_n=x_nu$ on rays, our experiments suggest
that suitable fixed points of $V_{z_n}$ approach nontrivial zeros of
$\zeta$ in a nearly logarithmic pattern. We formulate this behavior as
conjectures with explicit restrictions on the ray. We prove a conditional
result identifying finite limits as zeros, and give sufficient hypotheses
for the convergence of refinement diagnostics.
\newline \newline
We formulated  our conjectures
after  computer experiments using 
\it Mathematica. \rm Data and 
\it Mathematica \rm notebooks
for this project are at
ResearchGate, here \cite{Br3}.
Several other writers
have considered the dynamics of 
the Riemann zeta function,
 for example,
Kawahira \cite{Ka} and Woon \cite{W}.
\subsection{Definitions.}
We fix the following notation
for the remainder of the article: 
\newline \newline
Objects denoted $z$ and $s$ are generic 
complex numbers.
\newline
$\mathbb{C}^{\it cut}:=\mathbb{C}\setminus(-\infty,0)$; in particular,
$0$ is included, and continuity at $0$ refers to this subspace.
\newline 
$\rho$ is a 
nontrivial Riemann zero.
\newline $V_z(s):=\zeta(s) \exp (zs)$.
Thus $V_0(s) = \zeta(s)$.
\newline
$\Phi_z$ is the set of fixed points of $V_z$. Thus $\Phi_0$ is the set of fixed points of zeta.
\newline $\psi$ 
is a fixed point of zeta.
\newline $\psi_{\rho}$ is
a nearest fixed point of zeta to $\rho$, whenever such a point exists.
Conjecture 1 asserts existence and uniqueness.
\newline $u$ is
a point on the
unit circle.
\newline
$\vec{R}_u$ is the ray
emanating from zero and passing through $u$.
\newline
$X=(x_0, x_1,... )$
is an infinite increasing
arithmetic  progression with 
$x_0 = 0$.
\newline 
$\vec{\Phi}_{X,u}:=
\Phi_{0} \times \Phi_{x_{_1} u} 
\times \Phi_{x_{_2} u} \times ...$ .
\newline
$\vec{\phi}$
is a member of
$\vec{\Phi}_{X,u}$.
We write $\vec{\phi} = 
(\phi_0, \phi_1, ...)$.
\newline \newline
Let $\vec z=(z_0,z_1,\ldots)$ converge to $\gamma$, with
$z_n\ne\gamma$. An argument lift is a real sequence $\vartheta_n$ such
that $e^{i\vartheta_n}=(z_n-\gamma)/|z_n-\gamma|$.
Our default \emph{discrete} convention chooses
$\theta_{\vec z}(z_0)=\operatorname{Arg}(z_0-\gamma)$ and, recursively,
chooses $\theta_{\vec z}(z_{n+1})$ to be the smallest representative
strictly greater than $\theta_{\vec z}(z_n)$.
This determines a unique lift once the initial argument is fixed.
It need not recover the continuous angle of an interpolating curve:
a step may omit complete turns, and clockwise motion can be represented
by positive increments. A different lift will always be specified.

Fix an integer $K\geq1$, independently of $h\geq0$, and put
$r_n=|z_n-\gamma|$. For a specified lift with at least two distinct
angles in each window, define $(m_{h,K},b_{h,K})$ as the unique minimizer
of
\[
 \sum_{n=h}^{h+K}\bigl(\log r_n-m\vartheta_n-b\bigr)^2,
 \qquad (m,b)\in\mathbb R^2.
\]
For all sufficiently large $h$, $0<r_n<1$ throughout the window, so
\[
 d_{h,K}(n)=
 \frac{|m_{h,K}\vartheta_n+b_{h,K}-\log r_n|}{|\log r_n|},
 \qquad d_{h,K}=d_{h,K}(h+K)
\]
is defined. We call the sequence \emph{nearly logarithmic with respect
to $\gamma$ and the specified lift} if, for every fixed $K\geq1$,
there are $C_K>0$, $0<q_K<1$, and $h_K$ such that
$d_{h,K}\leq C_Kq_K^h$ for $h\geq h_K$.
Constants may depend on the sequence and on $K$.
Without an explicit alternative, the discrete convention is intended.
These are exact definitions; a finite-precision least-squares calculation
is only a numerical approximation to them.

For the same lift, set
\[
 \delta_n=|\vartheta_{n+1}-\vartheta_n|,
 \qquad \Delta_n=|\delta_{n+1}-\delta_n|.
\]
We call the sequence \emph{nearly uniformly distributed} if
$\Delta_n\leq Cq^n$ for some $C>0$, $0<q<1$, and all sufficiently
large $n$. This terminology concerns angular increments, not uniform
sampling with respect to arc length. In particular, a zero limiting
increment is not excluded by this definition.

For geometric spirals, we use a continuous angle rather than the
minimal discrete lift. A spiral with center $\gamma$ is the image of a
continuous injective curve $c:[0,\infty)\to\mathbb C\setminus\{\gamma\}$
with $c(t)\to\gamma$, whose radius decreases strictly and whose
continuous argument lift $\Theta(t)$ is strictly monotone and unbounded
in absolute value on a tail. Either direction of rotation is allowed.
Call such a spiral \emph{nearly logarithmic} if each inward sample
$z_n=c(t_n)$, $t_n\to\infty$, with
\[
 0<a\leq |\Theta(t_{n+1})-\Theta(t_n)|\leq b<2\pi
\]
on a tail, is nearly logarithmic using the inherited lift
$\vartheta_n=\Theta(t_n)$. Here $a,b$ may depend on the sample.
This sampling restriction makes the rate in the sample index meaningful;
the inherited lift prevents loss of turns.

An \emph{exactly logarithmic spiral} satisfies
\[
 \log|c(t)-\gamma|=m\Theta(t)+b,\qquad m\ne0,
\]
with $m\Theta(t)\to-\infty$ inward. Its least-squares residuals vanish
for every admissible sample, so it is nearly logarithmic in the preceding
sense. A discrete fit alone does not establish that a geometric spiral
with these properties exists. In \cite{Br2}, we studied deviations of
certain numerical spirals from fitted logarithmic spirals; here the
continuous and discrete assertions are kept separate.

\subsection{A theorem on limits 
of sequences from
\texorpdfstring{$\Phi_{X, u}$}
{TEXT}.} 
Before we list the conjectures,
we prove a theorem.
\begin{theorem}\label{thm:limits}
Let $|u|=1$ and let $X=(x_n)_{n\geq0}$ be an increasing arithmetic
progression of real numbers with $x_0=0$ and positive common difference.
Suppose $\phi_n\in\Phi_{x_nu}$, $\phi_n\to\lambda\in\mathbb C$, and
$\Re(u\lambda)>0$. Then $\zeta(\lambda)=0$. If in addition
$0<\Re\lambda<1$, this is a nontrivial zero.
\end{theorem}
\begin{proof}
The fixed-point equation is
$\phi_n=\zeta(\phi_n)e^{x_nu\phi_n}$, hence
\[
 |\zeta(\phi_n)|=|\phi_n|e^{-x_n\Re(u\phi_n)}.
\]
Convergence makes $|\phi_n|$ bounded and gives
$\Re(u\phi_n)\geq c>0$ for all sufficiently large $n$.
Since $x_n\to\infty$, the right-hand side tends to zero.
The limit cannot be $1$, since $\zeta$ has a pole there and its
modulus tends to infinity on approach to $1$.
Thus $\zeta$ is continuous at $\lambda$ and $\zeta(\lambda)=0$.
\end{proof}

\subsection{Conjectures.}
The ray conjectures below concern \emph{admissible pairs} $(u,\rho)$:
\[
 |u|=1,\quad u\ne-1,\quad
 \bigl(u=1\ \hbox{or}\ \Im\rho\,\Im u<0\bigr),\quad
 \Re(u\rho)>0,\quad \Im(u\rho)\ne0.
\]
The sign condition records the range originally considered; it is not
asserted to be necessary. The positive-real-part condition is essential
to this formulation: if $\phi_n\to\rho$, then
$|\zeta(\phi_n)|=|\phi_n|e^{-x_n\Re(u\phi_n)}$ rules out
$\Re(u\rho)<0$, because $\rho\ne0$. The boundary
$\Re(u\rho)=0$ is not addressed. The last restriction excludes
nonrotating leading behavior from the spiral conjectures.
The earlier sign condition alone is insufficient: for
$\rho=\beta+i\gamma$ with $\beta,\gamma>0$ and
$u=(-1-i\varepsilon)/\sqrt{1+\varepsilon^2}$,
$0<\varepsilon<\beta/\gamma$, it holds but $\Re(u\rho)<0$.
Conjecture 1 has no ray restriction. All later references to a ray
conjecture or proposal use admissible pairs.

\begin{conjecture} \rm 
Let $\Phi_z$ be the set of fixed points of $V_z$.
\newline
(1) The  set of imaginary parts 
$\{\Im \phi: \phi \in \Phi_{z} \}$ 
is unbounded
and nonempty.
\newline 
(2) Given a Riemann zero
$\rho$, there is a unique zeta fixed point
$\psi_{\rho}$ such that,
given any zeta fixed point $\psi_1\ne\psi_{\rho}$,
$|\psi_{\rho} - \rho| <|\psi_1 - \rho|$.
\end{conjecture}
\begin{conjecture}\label{conj:sequences}\rm
For every admissible pair $(u,\rho)$ and every progression
$X=(0,d,2d,\ldots)$, $d>0$, there is a sequence
$\vec\phi\in\vec\Phi_{X,u}$ converging to $\rho$ that is nearly
logarithmic and nearly uniformly distributed with respect to $\rho$,
using the discrete argument convention.
\end{conjecture}
This is an existence assertion for suitably selected fixed points.
The numerical evidence does not justify the stronger claim that every
possible choice converging to $\rho$ has these properties.
\begin{conjecture}\label{conj:curve}\rm
For each nontrivial zero $\rho$ there is a continuous function
$f^{cut}:\mathbb C^{cut}\to\mathbb C$ such that, for every admissible
pair $(u,\rho)$, its restriction $f$ to $\vec R_u$ satisfies:
\begin{enumerate}
\item $f:\vec R_u\to\mathbb C$ is continuous and one-to-one;
\item $f(z)\in\Phi_z$ for every $z\in\vec R_u$;
\item $f(0)=\psi_\rho$;
\item $f(tu)\to\rho$ as $t\to\infty$;
\item $S_{u,\rho}=f(\vec R_u)$ is a nearly logarithmic spiral
with center $\rho$, using the geometric definition in Section 1.1;
\item for every $X=(x_n)$, the values $\phi_n=f(x_nu)$ give a
sequence $\vec\phi_{u,\rho}\in\vec\Phi_{X,u}$ starting at
$\psi_\rho$ and converging to $\rho$.
\end{enumerate}
The curve and $f^{cut}$ are independent of $X$. The simultaneous
extension $f^{cut}$ across rays is an additional conjectural assertion;
experiments on a single ray do not establish it.
\end{conjecture}
\hskip -.2in
\textsection  \thinspace 
The next conjecture is
motivated by the analogy mentioned in the 
introduction and discussed later in the article.
\begin{conjecture}\label{conj:inverse}\rm
For each $X$ and $u$ there is a many-to-one function
$G=G_{X,u}$ on $\mathbb C$, independent of $\rho$, with the following
properties for every admissible $(u,\rho)$ and the values in
Conjecture \ref{conj:curve}:
\begin{enumerate}
\item $G(\phi_n)=\phi_{n-1}$ for $n\geq1$;
\item $G(\rho)=\rho$, the complex derivative $G'(\rho)$ exists,
and $|G'(\rho)|>1$;
\item there is an open connected set $U_\rho$ containing
$S_{u,\rho}\cup\{\rho\}$ and a map
$F_\rho:U_\rho\to U_\rho$ such that
\[
 G\circ F_\rho=\operatorname{id}_{U_\rho},\qquad
 F_\rho(\rho)=\rho,\qquad
 F_\rho(\phi_n)=\phi_{n+1},\qquad
 F_\rho(S_{u,\rho})\subset S_{u,\rho}.
\]
\end{enumerate}
Thus $F_\rho$ is injective and
$F_\rho\circ G=\operatorname{id}$ on $F_\rho(U_\rho)$ only.
The domain and map may depend on $X,u,\rho$.
Existence of this inverse map on a neighborhood of the whole spiral
is part of the conjecture.
\end{conjecture}
\begin{remark}
Obviously, it is 
 not surprising 
 that the $V_z$ might in some
way pick out
zeta zeros and fixed points
as special.
On the other hand, 
our numerical methods make it
feasible
to estimate the 
parameters incorporated in
the equations of
simple curves interpolating 
$V_z$ fixed points. 
This is possible because the parameter $z$ can be sampled on finer
grids and the resulting fixed points appear to follow continuous curves, something
we were unable to do for
analogous curves we studied in \cite{Br2},
where to do so would have required 
doing something else we do not know
how to do: extending
the iteration of zeta to non-integer heights.
(In the notation we will introduce below,
it would have required evaluating
expressions of the form
$\zeta^{\circ q}(s)$
for non-integer values of $q$.
There is some literature around the
problem of extending the
iteration operator, \it e.g.\rm, 
\cite{N2, N1}, but we have not
reduced those
results to code. In the
situation of the present article,
we do not have to
confront this difficulty.)
In the present situation,
conjecturally, 
these curves are nearly logarithmic 
spirals with 
a Riemann zero $\rho$ 
and the zeta fixed point 
$\psi_{\rho}$ 
as center and initial point, 
respectively. At present,
we are able to make
the estimates only for
one $\rho$
at a time. But
more study of these
estimates may be a way
to search for a
dictionary between
zeta zeros and zeta fixed points.
The Riemann hypothesis might
then be rephrased as a 
claim about the fixed
points. In another direction,
the convergence properties
of sequences from the
$\Phi_z$
bear upon the 
Riemann hypothesis
as well. (See question 1 below.)
\end{remark}
\begin{remark}
Conjecture 3 
(in which $\psi = \psi_{\rho}$) 
is
supported by some
experimental evidence
(see below). Based on the
analogy with the
situation of \cite{Br2} 
promised in the 
introduction and discussed in
the next section,
it seems plausible 
that conjecture 3
might extend to
arbitrary $\psi$.
To find experimental evidence
for other $\psi$
(corresponding to
evidence for the
analogous claim in \cite{Br2}),
we  would first need to identify 
the unknown
zeta-analogue $G$, because
the needed sequences in \cite{Br2} 
corresponding to the
$\vec{\phi} \in \vec{\Phi}_{X,u}$
in  conjectures  3 and 4
were obtained by  solving equations 
involving zeta.
Conjecture 4 is also
a claim about $\psi_{\rho}$,
since it involves the
curve $S_{u, \rho}$ and 
clause (6) of conjecture 3
associates $\psi_{\rho}$
to $S_{u, \rho}$.
Therefore the question arises
of the extension
of conjecture 4 to
arbitrary $\psi$.
Conjecture 4 is
based entirely on the analogy
with the situation of \cite{Br2}.
Thus, the same requirements
for finding experimental
evidence for its extension
to arbitrary $\psi$ apply
to conjecture 4 itself.
\newline \newline
We write down the
contemplated extensions 
of
conjectures 3 and 4
(labeled more tentatively there as
``proposals'')
in more detail in the 
appendix.
\end{remark}
\vskip .05in
\hskip -.2in
\textsection  \thinspace 
When iteration of
a function $f$ is meaningful,
we write
$f^{\circ 0}(s) = s$ and, for  
$n$ a positive integer,
$f^{\circ n}(s) = f(f^{\circ (n -1)}(s))$.
\begin{remark}
Assuming Conjectures \ref{conj:curve} and \ref{conj:inverse},
$F_\rho^{\circ n}(\psi_\rho)=\phi_n=f(x_nu)\to\rho$.
This gives the existence of a convergent sequence in
$\vec\Phi_{X,u}$. The stated convergence comes from Conjecture
\ref{conj:curve} and the identities advancing these values; it is not
deduced from $|G'(\rho)|>1$.
The discrete diagnostics in Conjecture \ref{conj:sequences} remain
separate assertions.
\end{remark}
\begin{conjecture} \rm
Here, it will be
convenient to indicate the
dependence of $G$ upon
$X$ and $u$ 
by writing $G = G_{X, u}$.
Suppose that
$X'$ is a refinement of  
$X$ in the sense that
the common difference  
between consecutive elements of  $X$ 
and between  consecutive elements of 
$X'$ are  $d, d'$ respectively
and  $d=Kd'$ for some
natural number $K$. Then 
$G_{X,u}(s)=G_{X',u}^{\circ K}(s)$ on their common domain of
finite, defined compositions (in particular, near each admissible
$\rho$).
\end{conjecture}
\subsection{Conjectures 4, 5
and the promised analogy.} 
In \cite{Br2}, we studied
the dynamical system of zeta
and several associated 
nearly logarithmic 
sequences
$\vec{\phi}_{\psi, \rho, \zeta} = 
(z_0, z_1, ...)$
such that \newline \newline
(*) $z_0 = \rho$, \newline
(**) $z_{n-1} = \zeta(z_n)$ for all 
$n > 0$,
and 
\newline
(***) 
$\lim \vec{\phi}_{\psi, \rho, \zeta} = \psi$.
\newline \newline
In that article, we
conjectured that such
a sequence existed
for each choice of the 
pair $(\psi, \rho)$,
but local attraction does not by itself prove this global assertion.
For a repelling fixed point $\psi$ with $|\zeta'(\psi)|>1$, the inverse
function theorem gives a local inverse $H$ fixing $\psi$, with
$H'(\psi)=1/\zeta'(\psi)$ and hence $|H'(\psi)|<1$. On a sufficiently small disk $U$ about $\psi$,
$H(U)\subset U$ and $H^{\circ n}(w)\to\psi$ for $w\in U$;
this is the local attracting behavior discussed in \cite{HY},
Theorem 2.6. To obtain a backward orbit beginning at a prescribed
$\rho$, one additionally needs a finite chain
\[
 z_0=\rho,\quad \zeta(z_j)=z_{j-1}\ (1\leq j\leq N),
 \quad z_N\in U.
\]
If such a chain exists, appending $H^{\circ k}(z_N)$ supplies the
required orbit. Existence of the finite chain is not a consequence of
local inversion, and the repelling hypothesis cannot be assumed for
all zeta fixed points (the real fixed point discussed in Section 2.3
is attracting). Thus this analogy supplies motivation, not a proof.

The behavior of iterates
of zeta were shown to
be similar in the range of our observations.
So we think it
is plausible to
propose a dynamical explanation
for the attraction of
the sequences in $\vec{\Phi}_{X,u}$
to Riemann zeros that draws on an analogy
with the proposals in \cite{Br2}.
In this analogy, in the 
context of \cite{Br2},
with 
 $X_n$ (say) $= (0, n, 2n, ...)$
$G$ should
correspond to $\zeta^{\circ n}$, 
the $F$'s
should
correspond to (complex analysis sense) 
branches of $\zeta^{\circ n (-1)}$,
and
the roles of the 
zeta fixed points $\psi$ and the
zeta zeros $\rho$ in \cite{Br2}
should be swapped.
This idea results in conjecture 4,
conjecture 5,
and proposal 2 (appendix.)
\newline \newline
We ask the following question for the real ray $u=1$.
\begin{question}
For every $X=(0,d,2d,\ldots)$ with $d>0$ and every
$\vec\phi\in\vec\Phi_{X,1}$, is it true that
$\Re\phi_n\to\sigma\in(0,1)$ implies $\sigma=\tfrac12$?
\end{question}

\section{\sc methods}
\subsection{Numerical specification and limitations.}
The supplied \textit{Mathematica} notebooks
\texttt{variants figs 5, 6redo.nb}, \texttt{variants fig. 9.nb}, and
\texttt{new variants fig10.nb} specify the calculations behind the
principal diagnostics. Their root-generation cells solve
\[
 \zeta(s)e^{xs}-s=0
\]
using \texttt{FindRoot}, initialized separately at $s=\rho_a$ for each
$x$. The settings in the root-generation cells are
\begin{center}
\begin{tabular}{lr}
\texttt{PrecisionGoal} & 200\\
\texttt{WorkingPrecision} & 600\\
\texttt{\$MinPrecision} (in function evaluations) & 400
\end{tabular}
\end{center}
The file \texttt{betterbarryzeros} contains the zero approximations.
These are requested numerical settings, not certified accuracy bounds.
In particular, a small equation residual does not by itself certify
which root was selected or bound its location error.

The integer-parameter generation cells use $x=1,\ldots,300$.
The fine-grid cells use $x=50.001,50.002,\ldots,100$.
They do not compute a continuation from $x=0$ in these runs.
The subsequently supplied data contain 2415 zero approximations,
30,000 integer-parameter records (300 for each of the first 100 zeros),
and all 50,000 fine-grid records in \texttt{run24jul17no1}, in order.
The later helper package \texttt{myFunctionsMathematica10.m} contains
the routines \texttt{lineParameters2} and \texttt{relDevWOLogs10}
used in the diagnostic calculations.

The Figure 9 notebook creates \texttt{run26jul17no5} from
\texttt{run24jul17no1}, adding $\log|\rho_1-\phi|$ and the increasing
argument lift of $\rho_1-\phi$. Thus the original processed file is not
needed to reconstruct the plotted quantities. For this revision, the
differences $\rho_1-\phi$ were formed with 180-decimal-digit arithmetic,
then their log distances, angles, fitted coefficients, and polygon
lengths were computed in double precision. For example, the slope at
filter index 1 is approximately $-0.00007975689356$, and the transformed
polygon length using all 50,000 vertices is approximately $24466.0069$.
These are reconstructions of finite numerical diagnostics; the original
root searches have not been rerun, and no certified error bounds for
the computed roots or diagnostics are asserted.
The inputs contain revised plotting cells as well as earlier variants;
no complete figure-to-run provenance or certified error analysis is
available. In particular, the longer run for Figure
\ref{seqfig8r1c1.png} and its precision are not specified by these
root-generation cells.

\subsection{Quadrant plots.}
We will make use 
of colored plots
(``quadrant plots'').
The  routine that makes
quadrant plots
takes as inputs
a specification 
of the image resolution,
the center and dimensions
of a region $R$
in the complex plane,
and a routine to compute some
$\mathbb{C} \rightarrow \widehat{\mathbb{C}}$ 
function
$f$. The
output colors a 
small square (the size of
which depends on the 
resolution) 
around each point $w$
of a regular lattice in $R$;
the color is chosen to
represent
the quadrant of $f(w)$.  
Similar methods
that do not use coloring are
put to use in, \it e.g\rm,
\cite{A}.)
\newline \newline
The pixel representing the square
is colored according to the rules in 
Table 1.
In the table,
the region $D$ is a disk with center $s = 0$ and 
large radius $r$ (chosen as may be convenient.) 
We denote the complement of $D$ as $-D$.
\newline \newline 
\hskip 1in
\begin{tabular}{|c|c|} \hline
Location of $f(s)$& Color of pixel depicting region containing 
$s$\\ \hline \hline
real and imaginary axes & black \\ \hline 
$D \hskip .05in\cap $ Quadrant I  &  rich blue\\ \hline 
$ - D \hskip .05in\cap$ Quadrant I  &  pale blue\\ \hline
$D \hskip .05in\cap $ Quadrant II  & rich red \\  \hline
$ - D \hskip .05in\cap$ Quadrant II  &  pale red\\ \hline
$D \hskip .05in\cap $ Quadrant III  & rich yellow \\ \hline 
$ - D \hskip .05in\cap$ Quadrant III  &  pale yellow\\ \hline
$D \hskip .05in\cap $ Quadrant IV  &  rich green\\ \hline
$ - D \hskip .05in\cap$ Quadrant IV  &  pale green\\ \hline 
\end{tabular} 
\vskip .1in
\hskip .5in
\sc{table 1: coloring scheme for quadrant plots}
\rm
\vskip .1in 
\hskip -.2in
The junction of four rich colors 
represents a zero,
the junction of four pale colors represents a pole,
and
the boundary of two appropriately-colored regions is an 
$f$ pre-image of an
axis. An example is shown in Figure \ref{leaf2point1.png}.
(We have superimposed a pair of axes on this 
quadrant plot.)
\begin{figure}[!htbp]
\centering
\includegraphics[scale=.8]{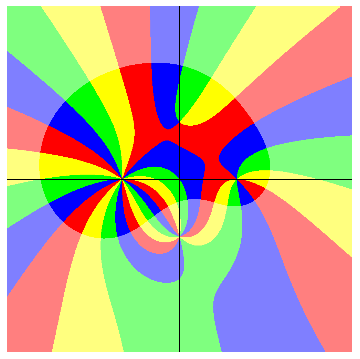}
\vskip .07in
\caption{Quadrant plot of $s \mapsto 
(s - 1)^2 (s - i) (s + 1)^5/(s+i)^3$}
\label{leaf2point1.png}
\end{figure}
\newpage
\subsection{Basins of attraction.}
For $z \in \mathbb{C} \rm \cup \{\infty\}$, 
let $A_z :=\{w \in \mathbb{C}$ s.t.
$\lim_{n \to\infty} 
\zeta^{\circ n}(w) = z\}$
(the ``basin of attraction'' of $z$ under 
zeta iteration.)
Only points for which all the iterates are defined are included in
these sets. Their intricate boundaries are illustrated numerically
in \cite{W}.
(As far as we know,
plots of  $A_{\infty}$
were made first by 
Woon in \cite{W}.)
Let $\phi \approx -.295905$
be the largest negative
zeta fixed point.
As we noted in \cite{Br2},
plots of
$A_{\phi}$ and the complement in
$\mathbb{C}$ of $A_{\infty}$
are indistinguishable to the eye.
These sets need not coincide: any finite periodic orbit other than
the fixed point $\phi$ belongs to the complement of $A_\infty$ but
not to $A_\phi$.
We reproduce plots of 
$A_{\phi}$ and the complement in
$\mathbb{C}$ of $A_{\infty}$
in Figure \ref{leaf1point1.png}
(appendix.)
\section{\sc Basis of the claims}
\subsection{Conjecture 1.}
\subsubsection{Existence and uniqueness of 
\texorpdfstring{$\psi_{\rho}$}{TEXT}.}
The right panel of Figure 
\ref{leaf3point2left.png}
(Figure 3.2 of \cite{Br2}) 
in the appendix is 
a quadrant plot
of the function 
$s \mapsto \zeta(s) - s$,
the zeros of which are the zeta fixed points.
The quadrant plot is shown superposed on
a plot of $A_{\phi}$.
The major feature
of this escaping set 
is that it consists of an infinite family of
irregularly shaped bulbs straddling
the critical strip, each
one of which 
(except for the central bulb, referred to
in \cite{Br2} as a
``cardioid'') contains one Riemann
zero (trivial or nontrivial,
but only the cardioid and 
the bulbs associated to
nontrivial zeros are visible at the scale
of Figure \ref{leaf3point2left.png}.) 
One zero of 
$s \mapsto \zeta(s) - s$
evidently lies on
one filament decorating
each of
the visible 
non-cardioid bulbs. 
These filaments
each consist of smaller bulbs
of the zeta escaping set, and
there is a numerical pattern
(described in \cite{Br2})
dictating the distribution of zeta 
fixed points among 
these bulbs. The fact that
this distribution is non-random
is evidence (we would argue)
that the association of
zeta fixed points to the Riemann zeros
is one-to-one, and, therefore,
that  there is probably
a unique $\psi_{\rho}$
for each $\rho$.
\subsubsection{The case
\texorpdfstring{$u=1$}
{TEXT}.} 
Clause 1 of conjecture 1 is the claim that  
the  set of imaginary parts 
$\{\Im \phi: \phi \in \Phi_z \}$ 
is unbounded
and nonempty.
The
set $\Phi_z$ is the set of solutions 
of the function 
$s \mapsto V_z(s) - s$.
When $z = 0$, 
$\Phi_z = \Phi_0$ is the set of fixed points
of the Riemann zeta function. 
Plots of these fixed points in \cite{Br2} 
provided our evidence for this case of 
the conjecture; we will not reproduce
them here. 
Using our particular
methods, we can only spot-check
this claim for selected $z$
and selected regions of the complex
plane. 
In Figure \ref{seqfig2.png} below, we display a
quadrant plot of the function 
$s \mapsto V_1(s) - s$
on a $2000$ by $2000$ square
with center at $s=0$ in the complex plane.
\begin{figure}[!htbp]
\centering
\includegraphics[scale=.8]{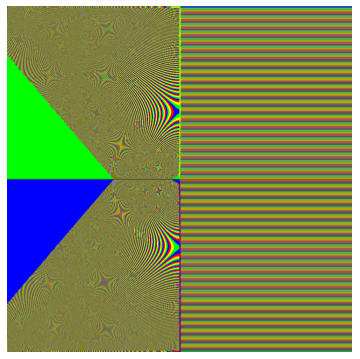}
\vskip .07in
\caption{2000 by 2000 quadrant plot of $s 
\mapsto V_1(s) - s$}
\label{seqfig2.png}
\end{figure}
Visible zeros of this function (lying at 
the junctions of four colors) lie in the vicinity
of the imaginary axis. There are also zeros (not visible
at this scale) in the vicinity of each nontrivial
Riemann zero in the depicted region,
on the boundary of a triangular region 
(barely
visible at this scale) near the center of the plot, 
and along the 
negative real axis.
The straight
boundaries of the green and blue
regions  do 
not (on inspection at smaller scales)
 contain any zeros of this function.
 \newline \newline
 In Figure \ref{seqfig3r1c1.png} below, we  display
quadrant plots of 
$s \mapsto V_k(s) - s, k = 2, 4, 8$
and  $16$, each centered at $s = 0$,
in squares with side length  
$400,  2400, 1.2 \times 10^5$, and  
$2 \times 10^8$ respectively.
They are shown clockwise from upper left in the 
order of the size of $k$.
 \begin{figure}[!htbp]
\centering
\includegraphics[scale=.45]{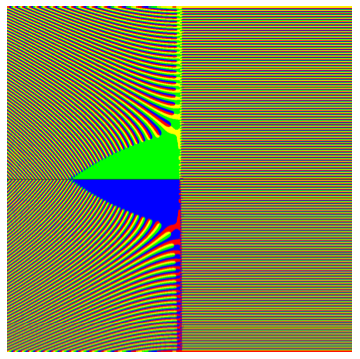}
\hskip .05in
\includegraphics[scale=.45]{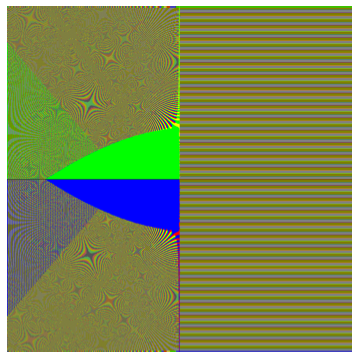}
\vskip .05in
\includegraphics[scale=.45]{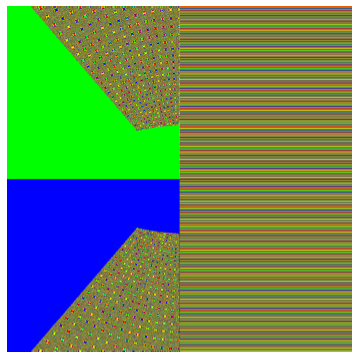}
\hskip .05in
\includegraphics[scale=.45]{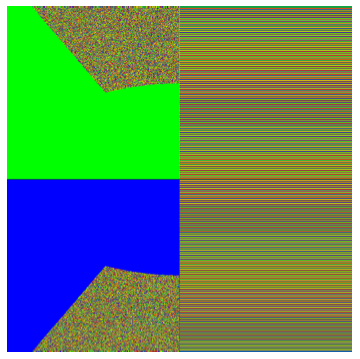}
\vskip .07in
\caption{Quadrant plots of $s 
\mapsto V_k(s) - s, k = 2, 4, 8$ and $16$}
\label{seqfig3r1c1.png}
\end{figure}
\newline \newline
Since the Riemann hypothesis has been
verified well beyond
the range of our experiments,
we are safe in designating the
$n^{th}$-by-height nontrivial zero  
in the upper half plane as $\rho_n$.
Figure \ref{seqfig4left.png} 
shows (at the four-color junctions)
zeros of 
$s \mapsto V_{100}(s) - s$ on two 
$1.2$ by $1.2$ squares, 
one centered at $\rho_1$
and the other centered at $\rho_{649}$.
The zeros along the 
left sides of the squares
apparently belong (as we will explain below) 
to sequences of fixed points  $\phi_n$
of the $V_n$, the real parts of 
which converge to zero
(but it is not clear that 
the imaginary parts converge at all.)
The zeros very near 
the centers of the squares
appear to belong
to such sequences converging to $\rho_1$
and $\rho_{649}$, respectively,
as we will also explain below.
 \begin{figure}[!htbp]
\centering
\includegraphics[scale=.4]{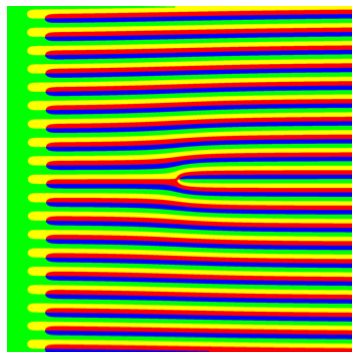}
\includegraphics[scale=.4]{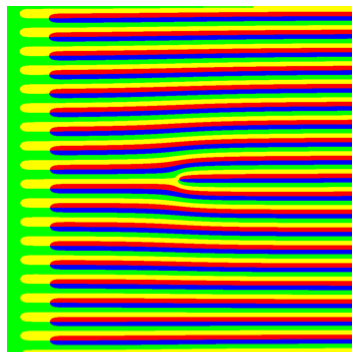}
\vskip .07in
\caption{Quadrant plots of $s \mapsto V_{100}(s) - s$
near $\rho_1$ and $\rho_{649}$.}
\label{seqfig4left.png}
\end{figure}
These and similar plots, which depict the 
fixed points of various $V_z$ as (apparently) 
unbounded sets of isolated points
in the complex plane,
are the basis of conjecture 1.
\subsection{Conjecture 2: the case
\texorpdfstring{$u=1$}{}.}
This is the
claim that, for $X$ and $\rho$
as above,
some  $\vec{\phi} \in \vec{\Phi}_{X,1}$
converges to $\rho$,
is nearly logarithmic with respect to
$\rho$, and is nearly uniformly 
distributed  with  respect to $\rho$.
We will describe experiments 
that tend to support this claim
for
 $X = (0, 1, 2, ...), X' = 
 (0, \frac 15, \frac 25, ...)$ 
 and for  
$\rho_n$ with $n$ selected from 
the range $1 \leq n \leq 100$. 
\subsubsection{Typical plots.}
The red points depict computed fixed points, conjecturally the
values $\phi_n=f(x_nu)$ in Conjecture \ref{conj:curve}.
For $\phi\ne\rho$, the notebook's displayed complex coordinate is
\begin{equation}\label{eq:display}
 T_\rho(\phi)=\log|\phi-\rho|\,
                  \frac{\rho-\phi}{|\rho-\phi|}.
\end{equation}
The displayed center is the coordinate origin, not $\rho$.
Thus $|T_\rho(\phi)|=|\log|\phi-\rho||$.
When $|\phi-\rho|<1$, the displayed point points in the direction
$\phi-\rho$ and moves outward as the original point approaches $\rho$.
When $|\phi-\rho|>1$, its direction is reversed.
Blue chords join consecutive displayed points. This signed
logarithmic transformation does not preserve Euclidean lengths.

For the log-radius-versus-angle panels, the notebook lifts
$\operatorname{Arg}(\rho-\phi_n)$ with the same minimal positive
increment rule as in Section 1.1; denote this lift by $\alpha_n$.
It differs from $\theta_{\vec\phi}(\phi_n)$ by one constant modulo
$2\pi$, with a consistent real representative. This changes the fitted
intercept but not the slope, angular increments, or residuals.
For the error panels, the routine \texttt{relDevWOLogs10} fits all
51 points $\phi_n,\ldots,\phi_{n+50}$ and returns the relative error
at the last point. Thus the plotted quantity is
\[
 h_n=d_{n,50}=
 \frac{|m_{n,50}\theta_{\vec\phi}(\phi_{n+50})+b_{n,50}
                 -\log|\phi_{n+50}-\rho||}
      {|\log|\phi_{n+50}-\rho||}.
\]
The plot uses $\log h_n$ and omits computed zero values.
The diagnostic cells request fitting precision 2000; this cannot add
accuracy to the root data. Their error panels provide finite numerical
evidence only for the displayed window size $K=50$, not the
all-$K$ asymptotic definition.

We  made plots  from which we  formulated 
conjecture 2. In one experiment,
we  
examined the  zeros 
$\rho_n, 1 \leq n \leq 100$.
Figures \ref{seqfig5r1c1.png} and 
\ref{seqfig6r1c1.png}  are typical outcomes
for $X = (0, 1, 2, ...)$. 
We chose them for display because,
together, they  indicate the dependence
upon $\delta_n$ of the  appearance
of the spirals.  
Figure  7  depicts several 
spirals centered on other zeros; 
we omit the  statistics
for these zeros, which are 
consistent with our conjectures.
 \begin{figure}[!htbp]
\centering
\includegraphics[scale=.45]{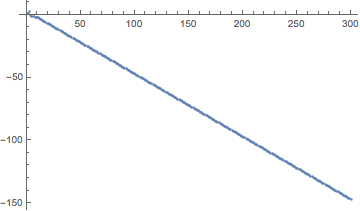}
\hskip .05in
\includegraphics[scale=.45]{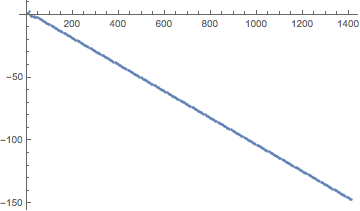}
\vskip .05in
\includegraphics[scale=.45]{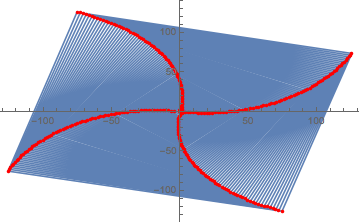}
\hskip .05in
\includegraphics[scale=.45]{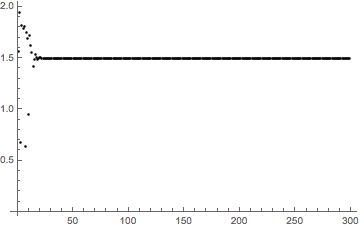}
\vskip .05in
\includegraphics[scale=.45]{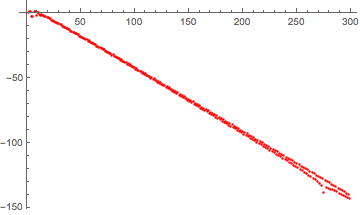}
\hskip .05in
\includegraphics[scale=.45]{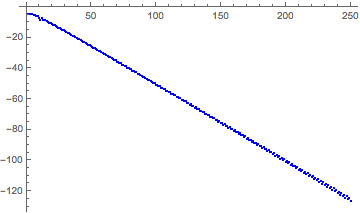}
\vskip .07in
\caption{Zero $\rho_{1}, u = 1, X  = (0, 1, 2, ...)$, 
$1 \leq n \leq 300$
\newline row 1, column 1:
$\log |\phi_n - \rho_{1}|$ vs. $n$
\newline  row 1, column 2: 
$\log |\phi_n - \rho_{1}|$ vs. $\alpha_n$
\newline row 2, column 1: logarithmically scaled plot
of  $\vec{\phi}$
\newline row 2, column 2: $\delta_n/\pi$ vs. $n$
\newline 
\newline row 3, column 1: $\log \Delta_{n-1}$ vs. $n$
\newline row 3, column 2:  $\log h_n$ vs. $n$ (K =  50)}
\label{seqfig5r1c1.png}
\end{figure}
\newpage
 \begin{figure}[!htbp]
\centering
\includegraphics[scale=.45]{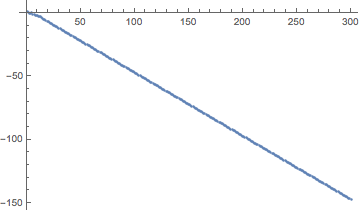}
\hskip .05in
\includegraphics[scale=.45]{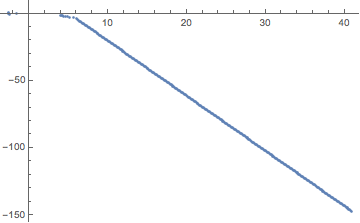}
\vskip .05in
\includegraphics[scale=.45]{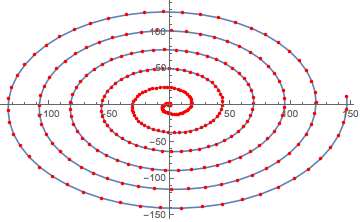}
\hskip .05in
\includegraphics[scale=.45]{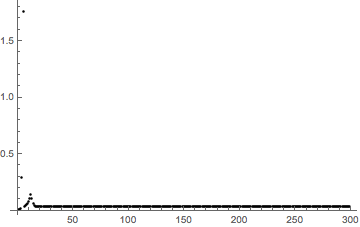}
\vskip .05in
\includegraphics[scale=.45]{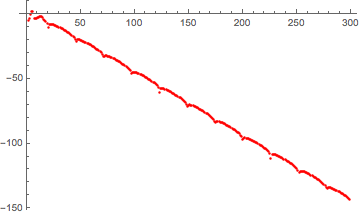}
\hskip .05in
\includegraphics[scale=.45]{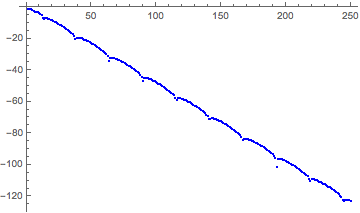}
\vskip .07in
\caption{Zero $\rho_3, u = 1, X  = (0, 1, 2, ...)$, 
$1 \leq n \leq 300$
\newline row 1, column 1:
$\log |\phi_n - \rho_{3}|$ vs. $n$
\newline  row 1, column 2: 
$\log |\phi_n - \rho_{3}|$ vs. $\alpha_n$
\newline row 2, column 1: logarithmically scaled plot
of  $\vec{\phi}$
\newline row 2, column 2: $\delta_n/\pi$ vs. $n$
\newline 
\newline row 3, column 1: $\log \Delta_{n-1}$ vs. $n$
\newline row 3, column 2:  $\log h_n$ vs. $n$ (K =  50)}
\label{seqfig6r1c1.png}
\end{figure}
 \begin{figure}[!htbp]
\centering
\includegraphics[scale=.40]{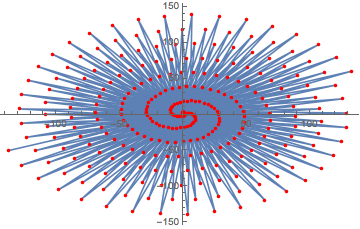}
\hskip .03in
\includegraphics[scale=.40]{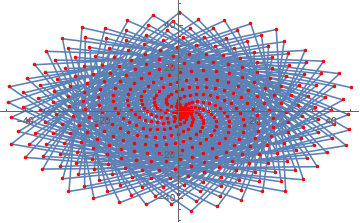}
\vskip .03in
\includegraphics[scale=.40]{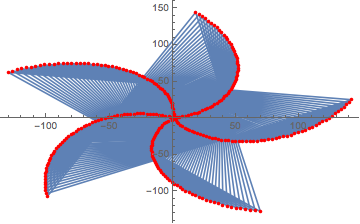}
\hskip .03in
\includegraphics[scale=.40]{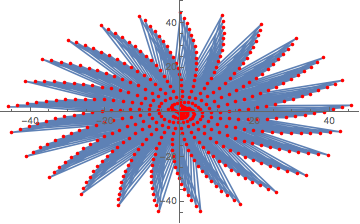}
\vskip .03in
\includegraphics[scale=.40]{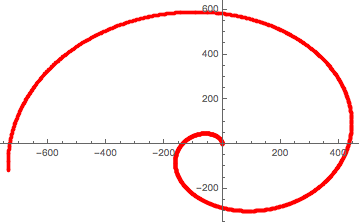}
\hskip .03in
\includegraphics[scale=.40]{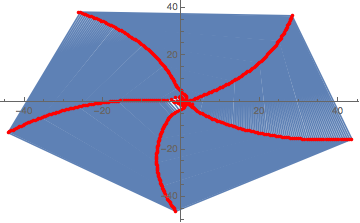}
\vskip .03in
\includegraphics[scale=.40]{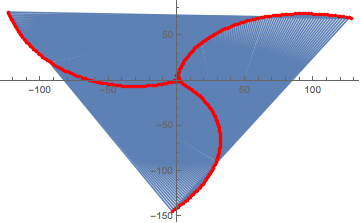}
\hskip .03in
\includegraphics[scale=.40]{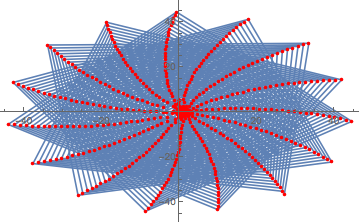}
\vskip .03in
\includegraphics[scale=.40]{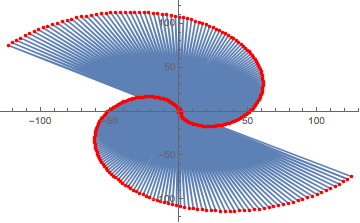}
\hskip .03in
\includegraphics[scale=.40]{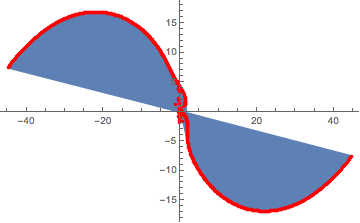}
\vskip .03in
\caption{Row 1, $\rho_7$; row 2, $\rho_{65}$;
row 3,  $\rho_{70}$; row 4, $\rho_{78}$,
row 5, $\rho_{82}$;\newline
column 1, $X = (0, 1, 2, ...)$; column 2, 
$X = (0, \frac 15, \frac 25, ...)$}
\label{seqfig7zr7left.png}
\end{figure}
\begin{figure}[!htbp]
\centering
\includegraphics[scale=.45]{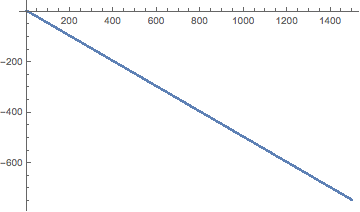}
\hskip .05in
\includegraphics[scale=.45]{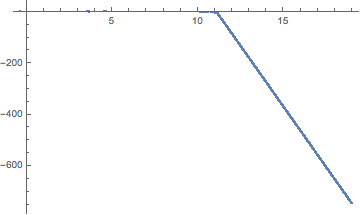}
\vskip .05in
\includegraphics[scale=.45]{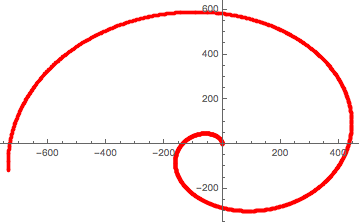}
\hskip .05in
\includegraphics[scale=.45]{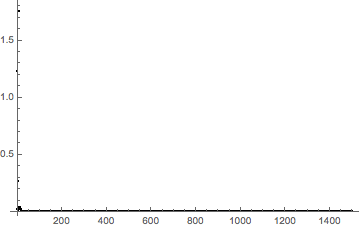}
\vskip .05in
\includegraphics[scale=.45]{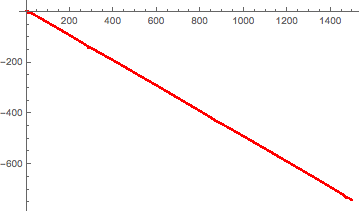}
\hskip .05in
\includegraphics[scale=.45]{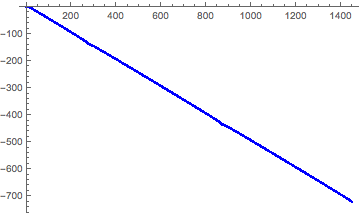}
\vskip .07in
\caption{Nontrivial zero $\rho_{70}, X  = (0, 1, 2, ...)$, 
$0 \leq n \leq 1500$;\newline
row 1, column 1:
$\log |\phi_n - \rho_{70}|$ vs. $n$;\newline
row 1, column 2: 
$\log |\phi_n - \rho_{70}|$ vs. $\alpha_n$;\newline
row 2, column 1: logarithmically scaled plot
of  $\vec{\phi}$;\newline
row 2, column 2: $\delta_n/\pi$ vs. $n$;\newline
row 3, column 1: $\log \Delta_n$ vs. $n$;\newline
row 3, column 2:  $\log h_n$ vs. $n$ (K =  50)}
\label{seqfig8r1c1.png}
\end{figure} \hskip -.1in
\subsubsection{An anomaly.}
For the zero
$\rho_{70}$, while the other
statistics we display in these 
plots were entirely consistent
with our conjectures, $\delta_n$
was so small in the range 
$0 \leq n \leq 300$ that the spiral 
structure for $X = (0, 1, 2, ...)$ 
was not obvious and
we had to extend our observations 
to the range $0 \leq n \leq 1500$
to see it. 
It appears that, for $\rho_{70}$, $\delta_n$
converges rapidly to a limit. One displayed revolution takes
approximately $2\pi/\delta_n$ steps; our numerical estimate is about
$1187$ steps per revolution. This estimate corresponds to an angular
increment of approximately $2\pi/1187\approx0.001685\pi$.
This concerns the discrete displayed angle and does not determine
how many turns a continuous interpolating curve makes between
consecutive values.
%see seqfig7.nb
We have not encountered 
another spiral like this. 
(The anomaly does not extend, for example,
to the
case $X = (0, \frac 15, \frac 25,...)$.)
The uniqueness (within the 
range of our observations) 
of this anomaly suggests that,
for the purpose of searching
for the unknown functions $G$,
one should look for an invariant of
the Riemann zeros that takes
an anomalous value at $\rho_{70}$.
As we said above,
we regard
the Riemann zeta function 
as analogous to $G$;
spirals  studied in 
\cite{Br2} are centered
upon zeta fixed points $\psi$,
and  $\lim_{n \to \infty} \delta_n$ 
appears to be determined by
$\arg \frac d{ds}\zeta(s)|_{s = \psi}$.
Therefore 
 we might expect
 the  invariant we are looking for
to
coincide with 
 $\arg \frac d{ds}G(s)|_{s = \rho_n}$,
which we might hope to identify 
without knowing $G$ explicitly.
We  have undertaken a 
cursory search for 
such an
invariant, so far without success.
The plots for
$\rho_{70}$ are in Figure \ref{seqfig8r1c1.png}. 
An interesting feature of the spiral
about $\rho_{70}$ for
$X = (0, 1, 2, ...)$ is that 
red points  $p_n$ and $p_{n'}$ on it
are close just if  $|n - n'|$
is small; on most  of the other
spirals we will display
in this article, this is not the case,
as one can  tell  by  keeping track of the
blue connecting chords. 
(The spiral about $\rho_3$
depicted in Figure \ref{seqfig6r1c1.png},
for which the $\delta_n$ are
also fairly small, is an 
exception.) In the plot
shown in row 3, column 1 of 
Figure \ref{seqfig7zr7left.png} 
(and reproduced in Figure 
\ref{seqfig8r1c1.png}),  
consecutive  $p_n$ are
so close that  these chords are
not visible.
\subsection{Conjecture 3.}
Conjecture \ref{conj:curve} asserts a continuous fixed-point curve,
independent of the sampling progression, starting at $\psi_\rho$ and
approaching $\rho$. Compatibility of values on rational grids does not
by itself establish a continuous extension. The following gives a
sufficient additional hypothesis on a fixed ray.

\begin{lemma}[Extension from compatible grids]\label{lem:extension}
Let $D\subset[0,\infty)$ be dense and contain $0$. Suppose compatible
grid values define one function $g:D\to\mathbb C$, uniformly
continuous on $D\cap[0,M]$ for every $M>0$, and
$g(t)\in\Phi_{tu}$ for $t\in D$. Then $g$ has a unique continuous
extension $f:[0,\infty)\to\mathbb C$ and $f(t)\in\Phi_{tu}$
for every $t\geq0$.
\end{lemma}
\begin{proof}
Uniform continuity and completeness of $\mathbb C$ give a unique
continuous extension on each compact interval; the extensions agree
on overlaps by density. If $t_j\in D$ tends to $t$, then
\[
 \zeta(g(t_j))=g(t_j)e^{-t_ju g(t_j)}.
\]
The right-hand side has a finite limit. Thus $f(t)\ne1$, since otherwise
the left-hand side would diverge in modulus at the pole of $\zeta$.
Continuity of $\zeta$ at $f(t)$ now proves the fixed-point equation.
\end{proof}
The lemma does not prove injectivity, spiral geometry, the limit at
infinity, or the simultaneous extension across rays required in
Conjecture \ref{conj:curve}. Each remains an additional requirement.

Figure \ref{seqfig9r1c1.png} studies finite coarsenings of one computed
grid in $50<x\leq100$, with spacing $0.001$ and $u=1$, $\rho=\rho_1$.
For filter index $q$, the notebook retains list positions
$1,1+q,1+2q,\ldots$ within the list. Two such grids are nested when
the coarser index is a multiple of the finer index; successive integers
$q$ do not generally yield nested grids. Even a nested family of
coarsenings of this fixed finite grid cannot have mesh tending to zero.
An infinite refinement argument requires new grids of smaller spacing.

The notebook reduces the stored angles modulo $2\pi$, then recomputes
the minimal increasing lift separately for each filtered list, starting
with the first reduced angle in $[0,2\pi)$. It fits
$\log r=m_q\alpha+b_q$ by \texttt{LinearModelFit}. The displayed
integer filter indices are $q=1,\ldots,99$. The lower panels show
$\log|m_q-m_{q+1}|$ and $\log|b_q-b_{q+1}|$, $1\leq q\leq98$.
Re-lifting after filtering can discard turns, so these angle coordinates
need not be restrictions of one continuous angular coordinate.
The observed stabilization is a finite diagnostic, not a proof that
an infinite sequence is Cauchy.

\begin{lemma}[Convergence of fitted coefficients]\label{lem:coefficients}
Let $(m_j,b_j)$ be fitted coefficients along an infinite refinement
sequence, expressed in one common angular coordinate with fixed
origin. Suppose
\[
 \sum_{j=0}^\infty
 \bigl(|m_{j+1}-m_j|+|b_{j+1}-b_j|\bigr)<\infty.
\]
Then $m_j\to m$ and $b_j\to b$. If the summands are bounded by
$Cq^j$, where $C>0$ and $0<q<1$, then
\[
 |m-m_j|+|b-b_j|\leq\frac{Cq^j}{1-q}.
\]
\end{lemma}
\begin{proof}
For $k>j$, telescope both differences and apply the triangle
inequality. The resulting bound is the tail of the stated convergent
series, so both sequences are Cauchy. Passing $k\to\infty$ gives the
tail bound, and summing the geometric series gives the displayed estimate.
\end{proof}
Merely having successive differences tend to zero is insufficient;
for example, the partial sums of $\sum_{j\geq1}1/j$ diverge.
The summability hypothesis is a substantive strengthening and has not
been verified by the finite plots.

Coefficient convergence alone does not identify a curve as a logarithmic
spiral. Here is a sufficient residual condition that supplies that step.
\begin{lemma}[Identification of a logarithmic arc]\label{lem:arc}
Let $c:[a,b]\to\mathbb C\setminus\{\rho\}$ be continuous, and let
$\Theta$ be a continuous lift of $\arg(c-\rho)$. Let $P_j$ be nested
finite partitions of $[a,b]$ with mesh tending to zero. Suppose
$m_j\to m$, $b_j\to b_*$, and
\[
 \max_{t\in P_j}
 |\log|c(t)-\rho|-m_j\Theta(t)-b_j|\longrightarrow0.
\]
Then $\log|c(t)-\rho|=m\Theta(t)+b_*$ throughout $[a,b]$.
If $m\ne0$ and $\Theta$ is strictly monotone, $c$ parametrizes an
arc of an exactly logarithmic spiral.
\end{lemma}
\begin{proof}
Every point of $\bigcup_jP_j$ belongs to all sufficiently fine
partitions. Passing to the limit proves the identity there.
The union is dense, and both sides are continuous, so the identity
holds everywhere. For $m\ne0$, strict monotonicity of $\Theta$
makes the radius strictly monotone and ensures injectivity.
\end{proof}
This lemma requires uniform \emph{absolute} residual control on the
partitions. The endpoint relative-error diagnostic $h_n$ does not
supply it. The finite data neither establish these hypotheses nor
prove exact logarithmic geometry. Likewise, fitting only $50<x\leq100$
does not connect this portion of the data to the initial fixed point
$\psi_\rho$ at $x=0$.

\begin{figure}[!htbp]
\centering
\includegraphics[scale=.45]{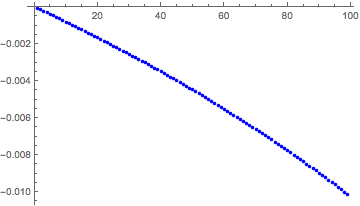}
\hskip .05in
\includegraphics[scale=.45]{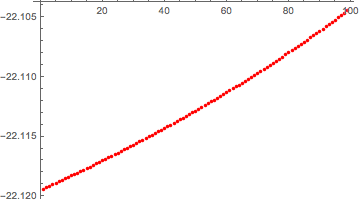}
\vskip .05in
\includegraphics[scale=.45]{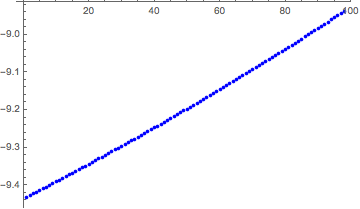}
\hskip .05in
\includegraphics[scale=.45]{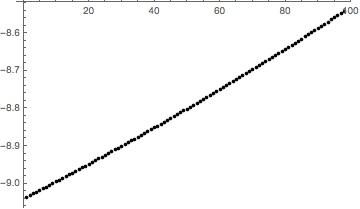}
\caption{Linear models of $\log r$ vs.
the recomputed angle $\alpha$ for $\rho=\rho_1$, $u=1$,
on filtered subsets of $(50.001,50.002,\ldots,100)$.
The horizontal coordinate is the filter index $q$.
\newline
row 1, column 1: slopes;
\newline 
row 1, column 2: intercepts;
\newline
row 2, column 1: $\log|m_q-m_{q+1}|$;
\newline
row 2, column 2: $\log|b_q-b_{q+1}|$
}
\label{seqfig9r1c1.png}
\end{figure}
\vskip .1in
\hskip -.23in
\thinspace
Figure \ref{seqfig10r1c1fi512.png} shows polygonal lines through
the transformed points $T_\rho(\phi)$ from \eqref{eq:display},
for filter indices $512,128,64,16$ in reading order.
These particular indices give nested vertex sets as the index decreases.
The third row shows sums of Euclidean chord lengths in the
\emph{transformed plane}, for indices $2^n$, $0\leq n\leq13$,
and $\lfloor1.11^n\rfloor$, $0\leq n\leq99$.
The latter indices are not generally nested; repeated indices can also
occur. The filtering routine retains the first point but need not retain
the last, so the terminal parameter can vary between polygons.
These finite plots do not give arc lengths in the original $s$-plane.

A rigorous length-convergence statement requires the following
additional hypotheses.
\begin{lemma}[Convergence of polygonal lengths]\label{lem:length}
Let $\Gamma:[a,b]\to\mathbb C$ be continuous and rectifiable.
Let $P_j$ be nested finite partitions containing both endpoints, with
mesh tending to zero. If
\[
 L_j=\sum_{[t_{k-1},t_k]\in P_j}
       |\Gamma(t_k)-\Gamma(t_{k-1})|,
\]
then $L_j$ increases to the length $L(\Gamma)$.
\end{lemma}
\begin{proof}
The triangle inequality gives $L_j\leq L_{j+1}\leq L(\Gamma)$.
For any finite partition $Q$, approximate its interior nodes by ordered
nodes of $P_j$. The mesh condition and continuity imply convergence of
the corresponding chord sum to the chord sum for $Q$; that approximating
sum is at most $L_j$. Hence $\lim_j L_j$ is at least the chord sum
for every $Q$. Taking the supremum over $Q$, which defines length,
gives the reverse inequality.
\end{proof}
Applied to this experiment, the lemma would require a continuous
rectifiable curve $\Gamma=T_\rho\circ c$ on a fixed compact parameter
interval, a genuinely infinite refinement, and inclusion of both
endpoints in every partition. These are additional hypotheses, not
conclusions from the plots. Stabilization of lengths alone does not
construct a limiting curve or prove that it is logarithmic.

\begin{figure}[!htbp]
\centering
\includegraphics[scale=.45]{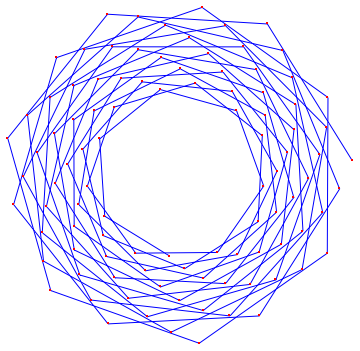}
\hskip .05in
\includegraphics[scale=.45]{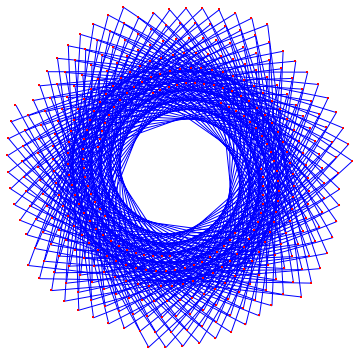}
\vskip .05in
\includegraphics[scale=.45]{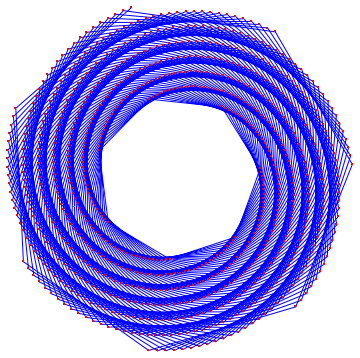}
\hskip .05in
\includegraphics[scale=.45]{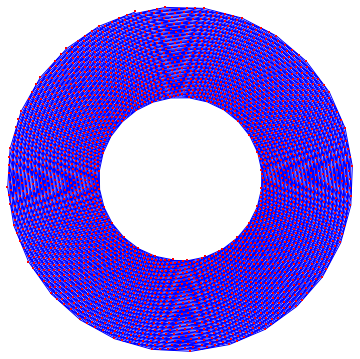}
\vskip .05in
\includegraphics[scale=.45]{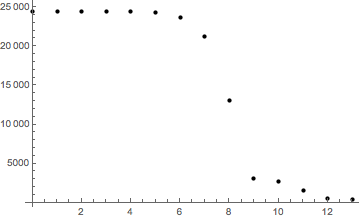}
\hskip .05in
\includegraphics[scale=.45]{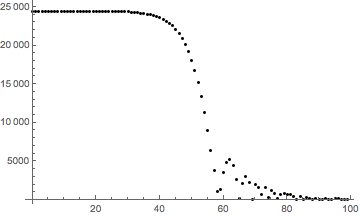}
\caption{$\rho = \rho_1, u = 1$,
filtered grid $(50.001,50.002,\ldots,100)$: 
rows 1, 2: logarithmically scaled polygons plotted against
filter index 512, 128, 64, and 16.
row 3: polygon lengths in the transformed plane for
filter indices $2^n$, $0\leq n\leq13$ (left), and
$\lfloor1.11^n\rfloor$, $0\leq n\leq99$ (right)
}
\label{seqfig10r1c1fi512.png}
\end{figure}
\clearpage
\section{\sc appendices}
\subsection{Extension of conjecture 3.}
\begin{proposal}\label{prop:curve}
For each nontrivial zero $\rho$ and each zeta fixed point $\psi$,
there is a continuous function $f^{cut}:\mathbb C^{cut}\to\mathbb C$
such that, for every admissible $(u,\rho)$, its restriction to
$\vec R_u$ has all the properties in Conjecture \ref{conj:curve},
with $\psi_\rho$ replaced by $\psi$.
Write its image as $S_{u,\psi,\rho}$ and its values as
$\phi_n=f(x_nu)$, forming $\vec\phi_{u,\psi,\rho}$.
The dependence of these values on $X$ is suppressed.
This remains a proposal; the numerical evidence for $\psi_\rho$ does
not establish it for arbitrary $\psi$.
\end{proposal}
\subsection{Extension of conjecture 4.}
\begin{proposal}\label{prop:inverse}
For each $X,u$ there is a many-to-one function $G$ on
$\mathbb C$, independent of $\psi$ and $\rho$, such that, for every
admissible $(u,\rho)$ and every curve and sequence of values postulated
in Proposal \ref{prop:curve},
\[
 G(\phi_n)=\phi_{n-1}\ (n\geq1),\qquad
 G(\rho)=\rho,\qquad |G'(\rho)|>1.
\]
Here the complex derivative $G'(\rho)$ is required to exist.
For each such $(\psi,\rho)$ there is an open connected set
$U_{\psi,\rho}$ containing $S_{u,\psi,\rho}\cup\{\rho\}$ and a
map $F_{\psi,\rho}:U_{\psi,\rho}\to U_{\psi,\rho}$ with
\[
 G\circ F_{\psi,\rho}=\operatorname{id}_{U_{\psi,\rho}},\qquad
 F_{\psi,\rho}(\rho)=\rho,\qquad
 F_{\psi,\rho}(\phi_n)=\phi_{n+1},
\]
\[
 F_{\psi,\rho}(S_{u,\psi,\rho})\subset S_{u,\psi,\rho}.
\]
Then $F_{\psi,\rho}$ is injective, and its inverse on its image is
the restriction of $G$. Existence of these domains and inverse maps
along the whole curves is part of the proposal.
\end{proposal}
\newpage
\subsection{Two basins of attraction.}
\begin{figure}[!htbp]
\centering
\includegraphics[scale=.45]{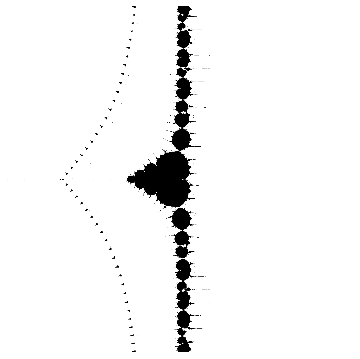}
\hskip .05in
\includegraphics[scale=.45]{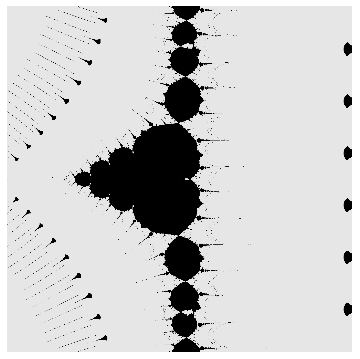}
\vskip .07in
\caption{left: $A_{\phi}$;
right: $\bf{C} - A_{\infty}$
}
\label{leaf1point1.png}
\end{figure}
\subsection{Figure 3.2 of \texorpdfstring{\cite{Br2}}{TEXT}.}
This figure is reproduced below as 
Figure \ref{leaf3point2left.png}.
\begin{figure}[!htbp]
\centering
\includegraphics[scale=.45]{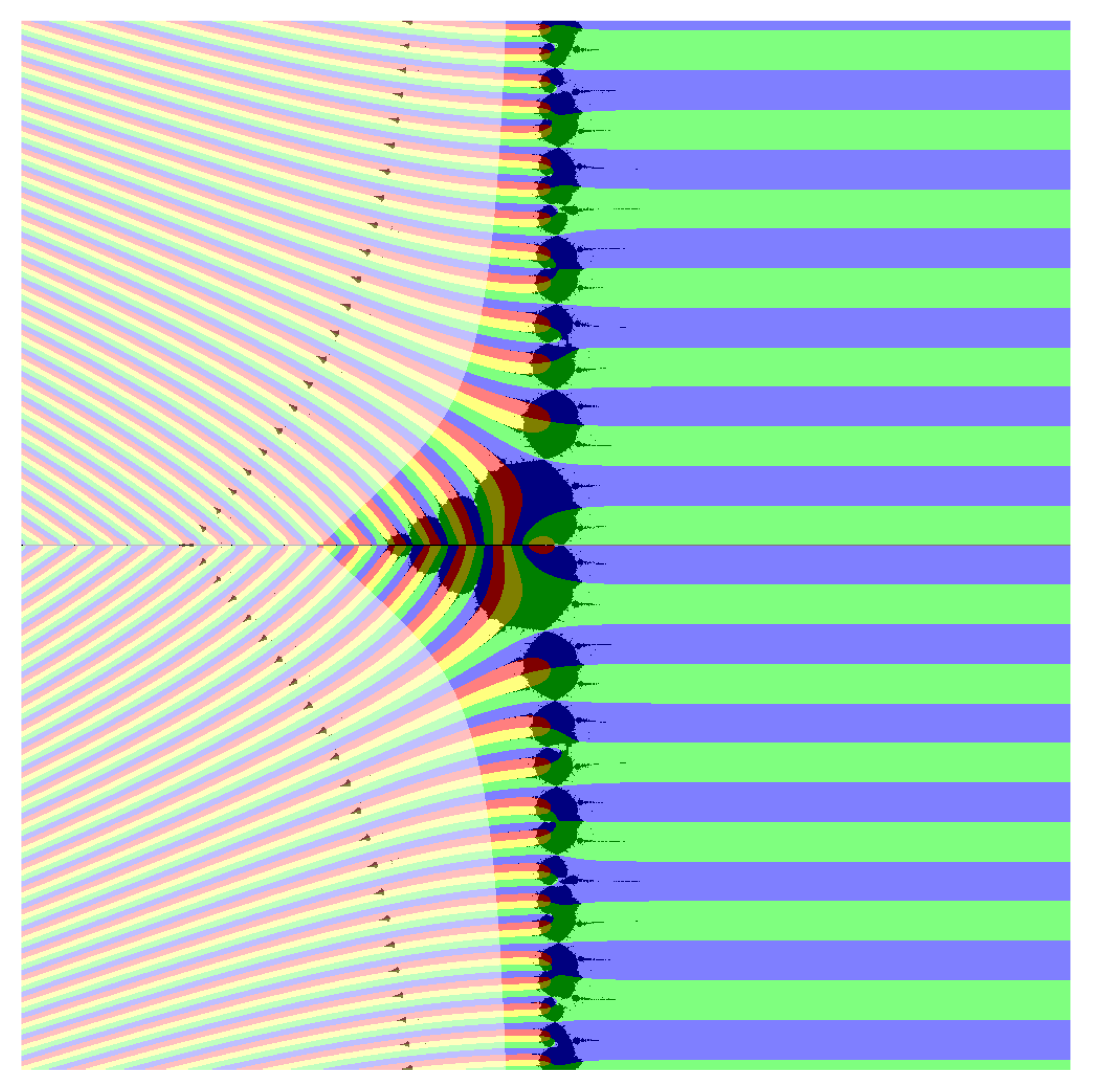}
\hskip .05in
\includegraphics[scale=.45]{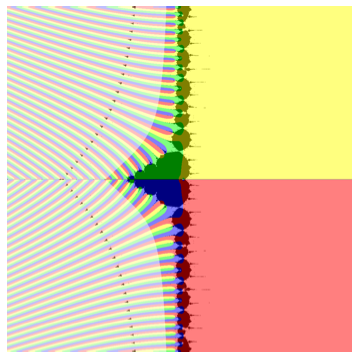}
\vskip .07in
\caption{left: zeros of zeta; right: zeros of
$s \mapsto \zeta(s) - s$ (\it i.e.\rm, zeta
fixed points);
both superposed on a plot 
of $A_{\phi}$
}
\label{leaf3point2left.png}
\end{figure}
\clearpage
\section{Disclosure of AI assistance in this draft}
\label{sec:ai-disclosure}
OpenAI's ChatGPT was used in preparing this revision for substantive
mathematical review, reformulation, proof drafting, numerical checking,
and LaTeX editing. The author supplied the original manuscript,
bibliography, figures, Mathematica notebooks, and numerical data.
The assistance extended beyond proofreading. The following account lists
all fifteen categories of changes made to the original manuscript.
Some changes repair errors while preserving the intended assertion;
others propose different definitions, stronger hypotheses, or weaker
conjectures. Their inclusion in this draft does not establish that they
are the only possible or preferred revisions.

\begin{enumerate}
\item \textbf{Qualification of the abstract and introduction.}
The original broad statement that fixed points converge to Riemann zeros
on nearly logarithmic spirals when the parameters tend to infinity was
replaced by a statement about selected fixed points along specified rays.
The revised wording distinguishes numerical observations and conjectures
from the conditional limit theorem. It also mentions sufficient
hypotheses for convergence of refinement diagnostics. This changes the
stated scope and certainty of the claims.

\item \textbf{Precision of the discrete nearly logarithmic definition.}
The incompatible requirements $K>h$ and $h\to\infty$ with fixed $K$
were replaced by a fixed window length $K\geq1$, independent of $h$.
Points equal to the limiting center were excluded. The fitted
coefficients were defined as exact least-squares minimizers, with the
nondegeneracy condition needed for uniqueness. The denominator in the
relative error was changed from $\log r_n$ to $|\log r_n|$, because
$\log r_n$ is negative near the center. Exponential decay was stated
explicitly as $d_{h,K}\leq C_Kq_K^h$, with $0<q_K<1$, eventually in $h$
for each fixed $K$. These changes separate the mathematical definition
from its finite-precision implementation.

\item \textbf{Reformulation of the geometric spiral definition.}
The distinction between the minimal increasing discrete argument and a
continuous argument along a curve was made explicit. The discrete
convention can omit complete turns and can represent clockwise motion
by positive increments. The replacement geometric definition uses a
continuous argument, permits either direction of rotation, and specifies
injectivity, a strictly decreasing radius throughout, and strictly
monotone angular behavior on a tail.
Admissible inward samples use the inherited continuous angle and satisfy
$0<a\leq|\Theta(t_{n+1})-\Theta(t_n)|\leq b<2\pi$ on a tail.
An exactly logarithmic spiral then has zero fitting residuals for these
samples. This is a substantive replacement definition, not a claim of
equivalence with the original wording.

\item \textbf{Replacement proof of Theorem 1.}
The original proof incorrectly used $\phi_n^*$ in place of
$\zeta(\phi_n^*)$ in the fixed-point equation. ChatGPT drafted a new
proof from
\[
 |\zeta(\phi_n)|=|\phi_n|e^{-x_n\Re(u\phi_n)}.
\]
Convergence of $\phi_n$ and the already stated hypothesis
$\Re(u\lambda)>0$ make the right-hand side tend to zero. The proof
excludes the pole $\lambda=1$ and then uses continuity of $\zeta$.
The conclusion was also clarified: the limit is a zero of $\zeta$;
it is nontrivial when it lies in the critical strip. The theorem's
positive-real-part hypothesis was retained, not newly introduced.

\item \textbf{Restrictions on the rays.}
The original restrictions were retained and supplemented by
$\Re(u\rho)>0$ and $\Im(u\rho)\ne0$. An explicit example was added
showing that the original sign restriction permits $\Re(u\rho)<0$,
which is incompatible with convergence of the stated fixed points to
$\rho\ne0$. This argument rules out negativity but does not settle the
boundary case $\Re(u\rho)=0$; strict positivity is an additional
restriction of the revised formulation. The nonzero-imaginary-part
condition is likewise a choice for the spiral conjectures, not a proved
necessary condition for all possible formulations. The original sign
restriction is described as the range considered experimentally rather
than asserted to be necessary.

\item \textbf{Weakening of Conjecture 2 and clarification of Conjecture 1.}
Conjecture 2 originally asserted the geometric properties for any
selection of fixed points converging to $\rho$. It now asserts the
existence of a suitably selected sequence having convergence and both
geometric properties. This is a weaker mathematical claim, chosen
because the experiments do not establish the assertion for arbitrary
selections; no counterexample to the original stronger conjecture was
proved in this revision. In Conjecture 1, the competing fixed point was
explicitly required to differ from $\psi_\rho$.

\item \textbf{Curve values and domains of inverse maps and compositions.}
Conjecture 3 now specifies the values as $\phi_n=f(x_nu)$, independence
of the curve from the progression $X$, and the distinction between
a curve on one ray and a simultaneous continuous extension across rays.
The notation $\phi_n$ was already present in the author's manuscript;
the explicit formula $\phi_n=f(x_nu)$ was introduced during AI-assisted
revision. At the author's request, the word ``samples'' introduced for
these values has been replaced by ``values'' in the conjecture and its
related discussion.

In Conjecture 4, the ambiguous inverse notation for a many-to-one $G$
was replaced by a map $F_\rho:U_\rho\to U_\rho$, where $U_\rho$ is
open and connected and contains the whole spiral and its limit.
The identities $G\circ F_\rho=\operatorname{id}$ on $U_\rho$ and
$F_\rho\circ G=\operatorname{id}$ on $F_\rho(U_\rho)$ were specified.
An earlier AI-assisted revision required this inverse map to be
holomorphic and asserted its multiplier and attraction at $\rho$.
The present revision removes the descriptors ``meromorphic'' for $G$
and ``holomorphic'' for $F_\rho$, and removes the multiplier and
attraction deductions and the appeal to local analytic inversion for
these conjectural maps. Existence of the complex derivative $G'(\rho)$
with $|G'(\rho)|>1$ is still required explicitly. Existence of the
inverse map along the whole spiral remains conjectural.
Corresponding changes and a corrected cross-reference were made in the
two appendix proposals. The identity in Conjecture 5 was restricted to
the common domain where its compositions are defined.

\item \textbf{Implications involving local dynamics.}
Remark 3 was rewritten to make its convergence assertion depend
explicitly on Conjecture 3 and the identities advancing the values,
rather than on the multiplier of $G$ alone. The valid local assertion
about zeta itself is retained: a repelling fixed point $\psi$ has a
local inverse fixing $\psi$ with derivative $1/\zeta'(\psi)$, and this
inverse attracts all points in a sufficiently small neighborhood.
The discussion of Hua and Yang's local attracting-point theorem was
also corrected: this local result does not automatically give a
backward orbit beginning at a prescribed $\rho$. A sufficient extra
condition was supplied, namely a finite backward chain from $\rho$
reaching the neighborhood on which the attracting inverse branch is
available. Local inverse iterates can then be appended to that chain.
This retained theoretical explanation concerns zeta, not the
conjectural maps $G$ and $F_\rho$.

\item \textbf{Four new conditional lemmas and their proofs.}
ChatGPT formulated and drafted the proofs of Lemmas
\ref{lem:extension}--\ref{lem:length}, using standard mathematical
arguments tailored to this manuscript. Lemma \ref{lem:extension}
requires compatible fixed-point values on a dense set and uniform
continuity on each compact interval; it obtains a continuous extension
and verifies that the fixed-point equation survives, including exclusion
of the pole. It does not establish injectivity, spiral geometry, a
limit at infinity, or extension across rays.

Lemma \ref{lem:coefficients} requires summability of the successive
changes of the two fitted coefficients. Telescoping proves convergence,
and a geometric bound $Cq^j$ gives the tail estimate
$Cq^j/(1-q)$. Lemma \ref{lem:arc} requires an existing continuous curve,
a common continuous angle, convergent coefficients, and uniformly
vanishing absolute residuals on nested partitions with mesh tending to
zero. These hypotheses give an exact logarithmic relation, and, with
nonzero slope and strictly monotone angle, a logarithmic spiral arc.
Lemma \ref{lem:length} assumes an existing continuous rectifiable curve
and nested partitions containing both endpoints, with mesh tending to
zero; the polygonal lengths then increase to the curve's length.

These lemmas give conditional repairs to the refinement arguments. Their
stronger hypotheses have not been established for the curves conjectured
in this article. In particular, the plotted endpoint relative errors do
not supply the uniform absolute-residual hypothesis of
Lemma \ref{lem:arc}.

\item \textbf{Interpretation of the refinement experiments.}
The assertion that successive integer filter indices produce nested
grids was corrected: selecting every $q$-th and every $(q+1)$-st point
does not generally do so; divisibility supplies the relevant nesting.
Coarsening a fixed finite grid cannot give an infinite sequence of
meshes tending to zero. Recomputing the angle separately after filtering
can discard turns. Small successive coefficient changes alone do not
prove the Cauchy property, and coefficient convergence alone does not
construct a limiting curve or prove that it is exactly logarithmic.
The discussion was rewritten to distinguish these conclusions from the
finite observations.

\item \textbf{Displayed coordinates, relative errors, and lengths.}
The impossible signed-radius equation
$|p-\rho|=\log|\phi-\rho|$ was replaced by the coordinate transformation
actually specified in the notebook,
\[
 T_\rho(\phi)=\log|\phi-\rho|\,
                  \frac{\rho-\phi}{|\rho-\phi|}.
\]
Its center in the displayed coordinates is the origin, and its modulus
is $|\log|\phi-\rho||$. The direction reversal and the outward
appearance of inward-moving points were explained. The previously
undefined $h_n$ was identified from the code as the relative log-radius
residual at the last point of a 51-point fitted window, $d_{n,50}$.
The notebook's angle based on $\rho-\phi_n$ was distinguished from the
angle based on $\phi_n-\rho$. Figure 10 was identified as measuring
polygonal lengths in the transformed plane; its filtering need not
retain a common terminal endpoint. These descriptions were derived
from the author's existing code.

\item \textbf{Captions, indices, and the revolution estimate.}
In Figures 5 and 6, the computed index range was corrected to
$1\leq n\leq300$, the notebook angle was denoted by $\alpha_n$, and
the difference-panel index was corrected to $\Delta_{n-1}$.
Erroneous references to $\rho_1$ were replaced by $\rho_3$ in Figure 6
and $\rho_{70}$ in Figure 8. Figure 9 now states the grid beginning at
$50.001$, the filter indices, and the quantities in its difference
panels. For Figure 10, ``fourth row'' was changed to ``third row,''
$1.1^n$ to $1.11^n$, and the corresponding exponent range to
$0\leq n\leq99$.

An earlier AI-assisted revision replaced the author's estimate of
about 1187 steps per revolution near $\rho_{70}$ by 1250, using only
the printed approximate increment $0.0016\pi$. That replacement is
withdrawn: an imprecisely printed increment does not justify replacing
the original revolution estimate. The present revision restores about
1187 steps and states its corresponding approximate increment
$2\pi/1187$. This restores the author's estimate; it does not claim a
new reproduction of the longer numerical run. The distinction between
displayed discrete turns and turns of a continuous interpolating curve
is retained.

\item \textbf{Numerical-methods documentation and additional checks.}
A subsection was added specifying the root-search starting values,
requested precision settings, parameter grids, data dependencies, and
fitting procedures. After the data were supplied, their record counts
and ordering were checked. Code written during this assistance
reconstructed the Figure 9 fitted coefficients and selected Figure 10
polygonal lengths. Differences $\rho_1-\phi$ were computed with
180-decimal-digit arithmetic, followed by double-precision diagnostic
calculations. The filter-index-1 slope and the polygon length for all
50,000 vertices were inserted as numerical examples.
The original root searches were not rerun, the roots were not certified,
and the full set of experimental claims was not independently reproduced.
Requested fitting precision was distinguished from the actual accuracy
of the underlying root data.

\item \textbf{The Riemann-hypothesis question and basin assertions.}
The question was explicitly restricted to $u=1$, with its sequences
quantified. An earlier AI-assisted revision added a paragraph about
possible implications for the Riemann hypothesis, distinguishing
convergence of real parts from finite complex convergence. At the
author's request, that entire paragraph immediately following Question 1
has now been removed. In the basin discussion, all iterates were
required to be defined. The unsupported
assertion about membership of zeta zeros in the basins was replaced by
the appropriate observation about other finite periodic orbits.
The unqualified assertion that the sets are fractals was replaced by a
description of the illustrated boundaries.

\newpage
\item \textbf{Editorial changes, LaTeX maintenance, and preserved material.}
The unused \texttt{ulem} package was removed, \texttt{amsmath} was added,
and package ordering was changed to correct duplicate hyperlink
destinations. Grammar and notation were corrected, a redundant repetition
of Conjecture 3 was removed, appendix headings were shortened, and page
placement was adjusted. The original title, author line, bibliography,
and PNG image contents were preserved. The figures were not regenerated
or replaced.
\end{enumerate}

Completion of these edits does not certify that the manuscript is free
of errors or prove its original conjectures. The replacement definitions,
reformulated conjectures, additional hypotheses, and drafted proofs
require the author's mathematical judgment and verification. The author
remains responsible for the claims and wording retained in the manuscript.

\clearpage
\bibliography{bibtexcite}
\href{mailto:barrybrent@member.ams.org}{barrybrent@member.ams.org}
\end{document}